\documentclass [11pt] {article}
\usepackage{amsfonts}
\usepackage{amssymb}
\usepackage{times}
\usepackage{graphicx}

\newtheorem{lemma}{Lemma}[section]
\newtheorem{corollary}{Corollary}[section]
\newtheorem{theorem}{Theorem}

\newtheorem{question}{Question}
\newtheorem{remark}{Remark}[section]

\newtheorem{proposition}{Proposition}[section]
\newtheorem{definition}{Definition}[section]

\def\Hom{\mbox{Hom}\,}

\def\<{\langle}
\def\>{\rangle}

\def\to{\rightarrow}

\begin{document}

\title{Sparse graph limits, entropy maximization and transitive graphs}
\author{{\sc Bal\'azs Szegedy}}

\maketitle

\abstract{In this paper we describe a triple correspondence between graph limits, information theory and group theory. We put forward a new graph limit concept called log-convergence that is closely connected to dense graph limits but its main applications are in the study of sparse graph sequences. We present an information theoretic limit concept for $k$-tuples of random variables that is based on the entropy maximization problem for joint distributions of random variables where a system of marginal distributions is prescribed. We give a fruitful correspondence between the two limit concepts that has a group theoretic nature. Our applications are in graph theory and information theory. We shows that if $H$ is a bipartite graph, $P_1$ is the edge and $t$ is the homomorphism density function  then the supremum of $\log t(H,G)/\log t(P_1,G)$ in the set of all graphs $G$ is the same as in the set of graphs that are both edge and vertex transitive. This result gives a group theoretic approach to Sidorenko's famous conjecture. We obtain information theoretic inequalities regarding the entropy maximization problem. We investigate the limits of sparse random graphs and discuss quasi-randomness in our framework.}

\section{Introduction}

In the frame of graph limit theory one considers large finite graphs as approximations of analytic objects and thus graph limit theory brings tools from analysis into graph theory. 
Quite interestingly, graph limit theory branches into a number of distinct theories depending on the number of edges in the graphs that we study. If the growth rate of the number of edges is quadratic in the number of vertices in a graph sequence then it is called a {\bf dense} graph sequence and in the sub-quadratic case it is called a {\bf sparse} graph sequence. The well established theory of dense graph limits (see: \cite{LSz},\cite{LSz2},\cite{BCLSV},\cite{L}), trivializes when applied for sparse sequences.
There are various limit theories for sparse graph sequences. Most of these limit theories are defined in the very sparse setting when graphs have bounded degree and in this case almost all limit concepts are variants of the so-called Benjamini-Schramm limit concept \cite{BS}. Despite of very promising directions \cite{CBCY},\cite{NP} the picture is even less coherent in the sub-quadratic but super-linear regime. The goal of this paper is to present a circle of new ideas in this subject that emerged as byproducts of the information theoretic approach \cite{Sz} of Sidorenko's famous conjecture \cite{Sid}.

For a pair of finite graphs $H,G$ let $t(H,G)$ denote the probability that a random function from $V(H)$ to $V(G)$ maps edges to edges. One can interpret $t(H,G)$ as the density of the graph $H$ in $G$. In dense graph limit theory a sequence of graphs $\{G_i\}_{i=1}^\infty$ is called convergent if $\lim_{i\to\infty} t(H,G_i)$ exists for every $H$.
Note that if $\{G_i\}_{i=1}^\infty$ is sparse then these limit numbers are all $0$.

Sidorenko's conjecture can be stated as the inequality $t(H,G)\geq t(P_1,G)^{|E(H)|}$ where $H$ is a bipartite graph , $P_1$ is the single edge and $G$ is an arbitrary graph. This was originally formulated by Sidorenko \cite{Sid} in an equivalent form as a family of correlation inequalities for Feynmann type integrals. The conjecture is verified for various families of bipartite graphs but a complete solution is still missing.

Sidorenko's inequalities are examples for graph inequalities that are linear after taking logarithm. An advantage of writing such inequalities in a logarithmic form is that the quantity $d(H,G):=-\log(t(H,G))$ has an information theoretic meaning that can be utilized in proofs. It was observed and exploited in \cite{Sz} that $d(H,G)$ is the relative entropy (KL-divergence) of the uniform distribution on edges in $G$ with respect to the uniform measure on $V(G)\times V(G)$. Entropy is usually measured in bits however quotients of the form $d(H_1,G)/d(H_2,G)$ are dimensionless quantities that are very natural to consider since they express the number $\alpha$ for which $t(H_2,G)^\alpha=t(H_1,G)$. (Note that the quantities $d(H_1,G)/d(H_2,G)$ are similar to homomorphism domination exponents however their behavior is different.) 

Roughly speaking, log-convergence is the convergence of all fractions $d(H_1,G_i)/d(H_2,G_i)$ in a graph sequence $\{G_i\}_{i=1}^\infty$. We have to be careful about a few things in this definition. The first problem is that these quantities are not always bounded and thus we loose the convenient compactness property that every graph sequence has a convergent sub-sequence. The second problem is that if $t(H,G)=0$ then $d(H,G)$ is not defined. There are various ways of getting around these problems (chapter \ref{remarks} is partially devoted to this issue) however if we work in the bipartite setting, as we do in most of the paper, then these problems disappear. In the bipartite setting graphs are equivalent with subsets in product sets $V_1\times V_2$. In this sense, from an algebraic point of view, the bipartite setting is more general than the graph setting since graphs are symmetric subsets of $V\times V$ and thus graphs can be regarded as special objects in the bipartite setting. For example Sidorenko's conjecture in the original form was formulated in the bipartite setting and it implies the analogous conjecture in the graph setting by regarding graphs as special objects in the bipartite setting. We differentiate between graphs in the bipartite setting and graphs that happen to be bipartite (for a more detailed explanation see chapter \ref{homchap}).

A convenient fact about the bipartite setting is that $1\leq d(H,G)/d(P_1,G)\leq c_H$ holds (if the edge sets of $G$ and $H$ are not empty) for some constant $c_H\leq |V_1(H)||V_2(H)|$ depending on $H$ where $V_1(H)$ and $V_2(H)$ are the two color classes in $H$. (Note that Sidorenko's conjecture says that the optimal value of $c_H$ is $|E(H)|$ but the weaker estimate $|V_1(H)||V_2(H)|$ is easy to prove.) This implies that (in the bipartite setting) every graph sequence contains a convergent sub-sequence since log-convergence is equivalent with the convergence of the quantities $h(H,G):=d(H,G)/d(P_1,G)$. 

Convergence of the quantities $d(H,G)$ is equivalent with dense graph convergence however the normalization by $d(P_1,G)$ changes the behavior significantly. Quite surprisingly log-convergence differentiates between an infinite family of sparse random graph models depending on a sparsity exponent $0<\beta\leq 1$. In these graph models edges in $G$ are created independently with probability $|V(G)|^{2\beta-2}$. In theorem \ref{quasi-thm} we determine the limiting quantities $h(H,G)$ (as $|V(G)|$  goes to infinity) in sparse random graph models depending on the parameter $\beta$ (and another parameter $\alpha$ that comes into the picture due to the bipartite setting and disappears in the graph setting). Our proof uses techniques developed for counting small sub-graphs in sparse random graphs \cite{B} and a special property of bipartite graphs. 

From the extremal combinatorics point of view there is a very convenient property of log-limits. Let $\mathcal{L}$ denote the completion of the set of (bipartite) graphs with respect to log-convergence. The graph parameters $G\rightarrow h(H,G)$ extend continuously to $\mathcal{L}$. The space $\mathcal{L}$ is compact and embeds naturally into $\mathbb{R}^\infty$ as a convex subset using the parameters $h(H,-)$ (this convexity is proved in lemma \ref{convex}). The Krein-Milman theorem implies that the log-limit space $\mathcal{L}$ is the closed convex hull of its extreme points. We can regard these extreme points as ergodic elements in $\mathcal{L}$.

Note that despite of the fact that graphons (two variable measurable functions representing dense graph limits) form a convex space there is no known natural convex structure on the dense graph limit space $\mathcal{W}$ consisting of equivalence classes of graphons. A large body of work in extremal combinatorics (in the dense setting) can be described as studying the properties of finite dimensional projections of the dense graph limit space using maps of the form $$W\rightarrow (t(H_1,W),t(H_2,W),\dots,t(H_k,W))\in\mathbb{R}^k$$ for a finite set of graphs $\{H_i\}_{i=1}^k$. These projections are compact but typically non convex and rather complicated shapes. Due to extensive research for decades there is a complete description of the two dimensional shape when $H_1$ is a single edge and $H_2$ is the triangle \cite{R}. However such a complete description is known only in a very few cases. Finite projections of the log-limit space $\mathcal{L}$ using  $h(H,-)$ are convex sets in $\mathbb{R}^k$ which gives hope for a nicer description using extremal points. 

Most of this paper deals with a fruitful correspondence between log-limits and an information theoretic limit concept for joint distributions of random variables. The information theoretic limit concept is based on an entropy maximization problem that is interesting on its own right. Quite surprisingly group theory comes naturally into the picture . 

Let us consider (finite) joint distributions $X=(X_1,X_2,\dots,X_k)$ of $k$ random variables. It is a classical fact that if we prescribe the individual distributions of $X_i$ for every $i$ then the joint distribution that maximizes the entropy with these marginals is the independent coupling of the given distributions. It is natural to investigate the more complicated entropy maximization problem in which we prescribe a system of marginal distributions of the form $\{X_i\}_{i\in L_j}$ where $L=\{L_j\}_{j=1}^n$ is a set system in $\{1,2,\dots,k\}$. In general it is not clear whether such a system of marginal constraints can be satisfied by any joint distribution at all. However if $L$ is the edge set of a bipartite graph $H$ and the marginal distributions are all the same, say $Y=(Y_1,Y_2)$, then there is at least one such joint distribution (see chapter \ref{emax}) and thus the entropy maximization problem makes sense. 
It turns out that the mutual information $d^*(H,Y)$ of the entropy maximizing distribution (which is unique) shares many properties with the logarithmic subgraph densities $d(H,G)$. It is worth mentioning that the entropy maximizing distribution is a Gibbs distribution and consequently a Markov random field on the vertices of $H$. 
We study the convergence notion corresponding to the normalized quantities $h^*(H,Y):=d^*(H,Y)/d^*(P_1,Y)$.
Convergence of the quantities $d^*(H,Y)$ is analogous to dense graph limits and convergence of $h^*(H,Y)$ is analogous to log-convergence. We say that a sequence of joint distributions $\{Y^i=(Y^i_1,Y^i_2)\}_{i=1}^\infty$ is $h^*$-convergent if $\lim_{i\to\infty}h^*(H,Y^i)$ exists for every bipartite graph $H$ with no isolated points.

A central result in this paper (see theorem \ref{main}) connects the parameters $h(H,-)$ and $h^*(H,-)$ through log-convergence.

\medskip

{\it For every finite joint distribution $Y=(Y_1,Y_2)$ there is sequence of graphs $\{G_i\}_{i=1}^\infty$ that are both edge and vertex transitive with $\lim_{i\to\infty} h(H,G_i)=h^*(H,Y)$.}

\medskip

We call graphs that are both edge and vertex transitive {\it edge-vertex transitive graphs}. (Note that in the bipartite setting automorphisms have to respect the color classes and so edge-vertex transitivity is equivalent with the property that the graph is edge transitive and has no isolated vertices.) Edge-vertex transitive graphs are fully described through the pair of stabilizers of the two endpoints of an edge and thus edge-vertex transitive graphs are given by triples $G,T_1,T_2$ where $G$ is a finite group and $T_1,T_2$ are subgroups in $G$. Subgraph densities of edge-vertex transitive graphs can be characterized through the number of solutions of equation system in finite groups and thus theorem \ref{main} puts the quantities $h^*(H,Y)$ into a group theoretic context.  

If $G$ is a graph and $X_G=(X_1,X_2)$ is a uniformly chosen random edge with endpoints $X_1$ and $X_2$ then we can apply theorem \ref{main} for $X_G$ and obtain a graph sequence $\{G_i\}_{i=1}^\infty$ of edge-vertex transitive graphs with $\lim_{i\to\infty} h(H,G_i)=h^*(H,X_G)\geq h(H,G)$. We can regard the graphs $G_i$ as uniformized (or smoothened) versions of $G$. Thus we encode valuable information from $G$ in highly symmetric and homogeneous objects. Using this correspondence we obtain a group theoretic and an information theoretic characterization of the values $c(H):=\sup_G h(H,G)$. Sidorenko's conjecture for a bipartite graph $H$ is equivalent with $c(H)=|E(H)|$. Since this is checked for various graphs $H$ we obtain new inequalities in group theory and information theory (see corollary \ref{cor-app1}.) On the other hand we also obtain that Sidorenko's conjecture holds for $H$ if and only if $t(H,G)\geq t(P_1,G)^{|E(H)|}$ holds in every edge-vertex transitive graphs $G$.

\section{Graph homomorphisms and dense graph limits}\label{homchap}

A graph homomorphism is a map from the vertex set $V(H)$ of a graph $H$ to the vertex set $V(G)$ of a graph $G$ such that edges are mapped to edges. Let $\Hom(H,G)$ denote the set of all homomorphisms. The (homomorphism) density of $H$ in $G$ is the probability that a random map from $V(H)$ to $V(G)$ is a homomorphism. We denote the homomorphism density by $t(H,G)$ and we have that $t(H,G)=|\Hom(H,G)||V(G)|^{-|V(H)|}$.

Graph homomorphisms can be studed in the context of bipartite graphs. Let $\mathcal{B}$ denote the set of finite graphs in which the vertices are partitioned into two classes labeled by the natural numbers $1$ and $2$ such that the endpoints of every edge have different label. If $G\in\mathcal{B}$ then we denote by $V_1(G)$ and $V_2(G)$ the partition classes given by the label. The edge set can be viewed as a subset in $V_1(G)\times V_2(G)$.
A homomorphism between two graphs in $\mathcal{B}$ is defined as a graph homomorphism with the extra property that it preserves the label of every vertex. The homomorphism density $t(H,G)$ inside $\mathcal{B}$ is defined as the probability that a random label preserving map from $V(H)$ to $V(G)$ is a graph homomorphism. As the next example shows, it is important to distinguish between graphs that happen to be bipartite and graphs in $\mathcal{B}$. Let $P_1$ be the single edge. One can calculate that $t(P_1,P_1)=1/2$. However if we view $P_1$ as an element in $\mathcal{B}$ with endpoints labeled by $1$ and $2$ then $t(P_1,P_1)=1$. 

\medskip

Homomorphis densities in both the general and in the bipartite contexts satisfies the following properties (see \cite{L}).

\medskip

\noindent{\bf Blow up invariance:}~ If $G_m$ is obtained from the graph $G$ by replacing each vertex by $m$-vertices and replacing each edge by the complete bipartite graph $K_{m,m}$ then $t(H,G)=t(H,G_m)$ holds for every $m\in\mathbb{N}$. In the bipartite setting, if $G_{m,n}$ is obtained from $G$ by replacing each vertex in $V_1(G)$ by $m$ points, each vertex in $V_2(G)$ by $n$ points and each edge by $K_{m,n}$ then $t(H,G)=t(H,G_{m,n})$.

\medskip

\noindent{\bf Right multiplicativity:}~ For two graphs $G_1,G_2$ let $G_1\times G_2$ denote graph with vertex set $V(G_1)\times V(G_2)$ and edge set $\{((v_1,w_1),(v_2,w_2))~|~(v_1,v_2)\in E(G_1),(w_1,w_2)\in E(G_2)\}$.  For two graphs $G_1$ and $G_2$ in $\mathcal{B}$ we define $G_1\times G_2$ the graph in $\mathcal{B}$ with $V_1(G_1\times G_2)=V_1(G_1)\times V_1(G_2)$ and $V_2(G_1\times G_2)=V_2(G_1)\times V_2(G_2)$. Edges are defined in the same way as in the non-bipartite setting by adding that $v_1,v_2\in V_1(G)$ and $v_2,w_2\in V_2(G)$. In both settings we have that $t(H,G_1\times G_2)=t(H,G_1)t(H,G_2)$.

\medskip

\noindent{\bf Left multiplicativity:}~ If $H_3$ is the disjoint union of $H_1$ and $H_2$ then $t(H_3,G)=t(H_1,G)t(H_2,G)$ holds for every $G$.

\medskip

\noindent{\bf One point graph:}~ If $P_0$ is the one point graph then $t(P_0,G)=1$ holds for every $G$. Note that in the bipartite setting there are two one point graphs up to isomorphism.  

\medskip

\noindent{\bf Monotonicity:}~ If $H'$ is defined on $V(H)$ and $E(H')\subseteq E(H)$ then $t(H',G)\geq t(H,G)$ holds for all graphs $G$.

\medskip

In the framework of the so-called dense graph limit theory, a sequence of graphs $\{G_i\}_{i=1}^\infty$ is called convergent if $\lim_{i\to\infty}t(H,G_i)$ exists for every finite graph $H$. Convergence in the bipartite setting can be defined in the same way. The limit of a convergent graph sequence can be represented by the {\it trivial limit object} which a graph parameter of the form $f:\mathcal{G}\rightarrow [0,1]$ where $\mathcal{G}$ is the set of (isomorphism classes of) finite graphs and $f(H):=\lim_{i\to\infty}t(H,G)$. Similarly, in the bipartite setting we get graph parameters of the form $f:\mathcal{B}\rightarrow[0,1]$ as trivial limit objects. Let $\mathcal{W}$ denote the set of all possible trivial limit objects for convergent graph sequences and let $\mathcal{W}_b$ denote the set of all possible trivial limit objects for convergent sequences in $\mathcal{B}$. It is clear that both $\mathcal{W}$ and $\mathcal{W}_b$ are closed compact sets in $\mathbb{R}^\infty$ with the product topology. However the structure of these sets is very far from being trivial. For example $\mathcal{W}$ and $\mathcal{W}_b$ are not convex. Projections of these sets to finitely many coordinates represented by finitely many graphs $\{H_i\}_{i=1}^k$ are very important in extremal graph theory since these finite dimensional shapes encode all possible inequalities between the densities of $\{H_i\}_{i=1}^k$. Even the simple looking case when $H_1$ is an edge and $H_2$ is the triangle took decades to completely describe. This two dimensional non-convex region has a boundary that is the union of countably many algebraic curves. 

\section{Edge-vertex transitive bipartite graphs}

In this paper we will need graph automorphisms in the bipartite setting. An automorphism of a bipartite graph $H\in\mathcal{B}$ is an ivertible homomorphism from $H$ to itself. In other words automorphisms in the bipartite setting are normal graph automorphisms with the extra condition that they preserve labels. We say that a bipartite graph $H\in\mathcal{B}$ is {\bf edge-vertex transitive} if it is both edge and vertex transitive. Note that in the bipartite setting $H$ is called vertex transitive if the automorphism group acts transitively on both $V_1(H)$ and $V_2(H)$. Edge-vertex transitivity in the bipartite setting is equivalent with the property that a graph is edge transitive and contains no isolated vertices.
The next definition and lemma shows that edge-vertex transitive graphs in $\mathcal{B}$ can be all described using only a pair of subgroups in a finite group and thus they are highly group theoretic objects. 

\begin{definition} Let $G$ be a finite group and let $T_1,T_2\leq G$ be subgroups in $G$. We denote by $\mathcal{G}(G,T_1,T_2)$ the graph $H$ in $\mathcal{B}$ such that $V_i(H):=\{gT_i~|~g\in G\}$ is the lef coset space according to $T_i$ for $i=1,2$ and $E(H)=\{(gT_1,gT_2)~|~g\in G\}$.
\end{definition}

\begin{lemma}\label{lem-bas1} The set of edge-vertex transitive graphs in $\mathcal{B}$ is the same as the set of graphs $\mathcal{G}(G,T_1,T_2)$ where $G,T_1,T_2$ are finite groups with $T_1,T_2\leq G$. 
\end{lemma}

\begin{proof} It is clear that every graph $\mathcal{G}(G,T_1,T_2)$ is edge-vertex transitive since the action $(gT_1,gT_2)^h:=(hgT_1,hgT_2)$ is transitive on the edges and on both left coset spaces. For the other direction let $H$ be an edge vertex transitive graph with automorphism group $G$ and let $(v_1,v_2)\in E(G)$ be a fixed edge. Let $T_i$ denote the stabilizer of $v_i$ for $i=1,2$. Then each vertex in $V_i(H)$ is uniquely determined by a left coset of $T_i$. The orbit of $(v_1,v_2)$ under the action of $G$ is the set of all edges and thus $H$ is isomorphic to $\mathcal{G}(G,T_1,T_2)$.
\end{proof}

\medskip

Note that $\mathcal{G}(G,T_1,T_2)$ is connected if and only if $T_1$ and $T_2$ generate the group $G$. It is also worth mentioning that there is a group theoretic interpretation of $t(H,\mathcal{G}(G,T_1,T_2))$ in terms of the number of solutions of an equation system in $G$ .
For a bipartite graph $H\in\mathcal{B}$ (with no isolated point) let $W(H,G,T_1,T_2)$ denote set of vectors $\{g_e\}_{e\in E(H)}$ in $G^{E(H)}$ satisfying 
$g_eg_f^{-1}\in T_i$ whenever $e\cap f\in V_i$. These equations express the fact that $g_eT_i=g_fT_i$ for every pair of edges $e,f$ with $e\cap f\in V_i$ and thus for every element $v\in V_i$ there is a unique coset $t_vT_i$ with the property that $g_eT_i=t_vT_i$ holds whenever $e$ contains $v$. This implies that the map $v\rightarrow t_vT_i$ (for $v\in V_i$) is a homomorphism of $H$ to $\mathcal{G}(G,T_1,T_2)$ and it is easy to see that every homomorphism is obtained in $|T_1\cap T_2|^{|E(H)|}$ ways. It follows that $$|\Hom(H,\mathcal{G}(G,T_1,T_2))|=|W(H,G,T_1,T_2)||T_1\cap T_2|^{-|E(H)|}$$ and thus
$$t(H,\mathcal{G}(G,T_1,T_2))=|W(H,G,T_1,T_2)||T_1|^{|V_1(H)|}|T_2|^{|V_2(H)|}|T_1\cap T_2|^{-|E(H)|}|G|^{-|V(H)|}$$

\section{Logarithmic graph limits}

The main motivation for our convergence notion comes from the study of graph theoretic inequalities that are linear in the logarithms of subgraph densities. It is well known for example that $t(C_4,G)\geq t(P_2,G)^2$ holds where $C_4$ is the $4$-cycle and $P_n$ is the path with $n$-edges. It was conjectured by Sidorenko that $t(H,G)\geq t(P_1,G)^{|E(H)|}$ holds whenever $H$ is bipartite. (This is conjectured in both in the bipartite and in the normal setting, but the bipartite version is stronger.) Sidorenko's conjecture is checked for a variety of graphs $H$. For a recent survey see \cite{Sz}.  These inequalities are all linear inequalities for the quantities $\log t(H,G)$. It is very natural to represent every graph $G$  by the graph parameter $H\mapsto -\log t(H,G)$ where the negative sign is used to get a non-negative number. It was pointed out in \cite{Sz} that $d(H,G):=-\log t(H,G)$ is the relative entropy (also called KL-divergence) of the uniform distribution on $\Hom(H,G)$ with respect to the uniform distribution on $V(G)^{V(H)}$. For studying linear inequalities between the quantities $d(H,G)$ it is enough to view the infinite dimensional vector $(d(H,G))_{H\in\mathcal{G}}$ up to a multiplication with scalar. In other words we wish to work in the infinite dimensional projective space. The loss of information by the projective view seems to be minor since we work with vectors in an infinite dimensional space and we loose basically one dimension. However this minor information loss turns out to be fundamental. It leads to a graph limit notion which is non-trivial for many interesting sparse graph sequences. We say that a graph sequence $\{G_i\}_{i=1}^\infty$ is {\bf log-convergent} if $\lim_{i\to\infty}d(H_1,G_i)/d(H_2,G_i)$ exists for every pair of graphs $H_1,H_2$ where both $H_1$ and $H_2$ have at least one edge. The limit here might be infinite. Another type of singularity that one has to be careful with is when $t(H_2,G)=0$ and thus $d(H_2,G_i)$ is not defined. It turns out however that in the bipartite setting we can completely avoid these infinities and thus our limit notion behaves nicer. In this paper we study our limit concept in the bipartite case and we will discuss the graph case in chapter \ref{remarks}. 

\begin{lemma}\label{kozott} For $H,G\in\mathcal{B}$ with $E(H)\neq\emptyset,E(G)\neq\emptyset$ we have that
$$d(P_1,G)\leq d(H,G)\leq c_H d(P_1,G)$$ for some constant $c_H$ depending on $H$.
\end{lemma}

\begin{proof} The inequality $d(P_1,G)\leq d(H,G)$ follows from $t(P_1,G)\geq t(H,G)$ which is a consequence of the monotonicity of $t$. The monotonicity of $t$ also implies that $d(H,G)\leq d(K,G)$ where $K$ is the complete bipartite graph on the vertex set $V(H)=V_1(H)\cup V_2(H)$. Since $K$ satisfies Sidorenko's conjecture \cite{Sid} we have that $t(K,G)\geq t(P_1,G)^{|V_1(H)||V_2(H)|}$ and thus $d(K,G)\leq |V_1(H)||V_2(H)|d(P_1,G)$. It follows that the statement of the lemma is satisfied with $c_H:=|V_1(H)||V_2(H)|$.
\end{proof}

Note that if $H$ statisfies Sidorenko's conjecture then $c_H=|E(H)|$ is the optimal choice in lemma \ref{kozott}.
Let $h(H,G):=d(H,G)/d(P_1,G)$. If $G$ is a complete graph then we have that $d(P_1,G)=d(H,G)=0$. In this case it is natural to define $h(H,G):=|E(H)|$ since this is the limit of $h(H,G_n)$ when $G_n$ tends to $G$ in the normalized cut norm. However if $G$ or $H$ has no edges (empty graph) there is no natural meaning of $h(H,G)$. Let $\mathcal{B}_0$ denote the set of graph $G$ in $\mathcal{B}$ such that $E(G)\neq\emptyset$. Note that lemma \ref{kozott} can also be written as $1\leq h(H,G)\leq c_H$ where $G\in\mathcal{B}_0$ and $H\in\mathcal{B}_0$.

\begin{lemma}\label{comp} A graph sequence $\{G_i\}_{i=1}^\infty$ in $\mathcal{B}_0$ is log-convergent if and only if $\lim_{i\to\infty} h(H,G_i)$ exists for every $H\in\mathcal{B}_0$. Every graph sequence in $\mathcal{B}_0$ has a log-convergent subsequence.
\end{lemma}

\begin{proof} If $\{G_i\}_{i=1}^\infty$ is log-convergent then by definition $h(H,G_i)$ is a convergent sequence if $E(H)\neq\emptyset$. On the other hand, by finiteness of limits, we have that $\lim_{i\to\infty}d(H_1,G_i)/d(H_2,G_i)=\lim_{i\to\infty}h(H_1,G_i)/h(H_2,G_i)=\lim_{i\to\infty}h(H_1,G_i)/\lim_{i\to\infty}h(H_2,G_i)$.
The second statement follows from $1\leq h(H,G)\leq c_H$.
\end{proof}

Similarly to dense graph limits we can represent convergent graph sequences by trivial limit objects. For a graph $G\in\mathcal{B}_0$ let $\tau(G)\in\mathbb{R}^{\mathcal{B}_0}$ denote the vector $(h(H,G))_{H\in\mathcal{B}_0}$.
A graph sequence $\{G_i\}_{i=1}^\infty$ in $\mathcal{B}_0$ is log-convergent if and only if $\{\tau(G_i)\}_{i=1}^\infty$ is a convergent sequence in the topological space $\mathbb{R}^{\mathcal{B}_0}$. The closure $\mathcal{L}$ of the set $\{\tau(G)\}_{G\in\mathcal{B}_0}$ is the graph log-limit space.  

\begin{lemma}\label{convex} The graph log-limit space $\mathcal{L}$ is a convex compact set in $\mathbb{R}^{\mathcal{B}_0}$.
\end{lemma}

\begin{proof} Let $x=\lim_{i\to\infty}\tau(G_i)$ and $y=\lim_{i\to\infty}\tau(K_i)$ for some log-convergent graph sequences $\{G_i\}_{i=1}^\infty$ and $\{K_i\}_{i=1}^\infty$ in $\mathcal{B}_0$. Let $0<\alpha<1$ be a real number. Let $L_i$ denote the graph $G_i\times G_i\times\dots \times G_i\times K_i\times K_i\times\dots\times K_i$ where $G_i$ is used $n_i$-times and $K_i$ is used $k_i$ times for some sequence $\{n_i\}_{i=1}^\infty$ and $\{k_i\}_{i=1}^\infty$ of natural numbers with $\lim_{i\to\infty}d(P_1,G_i)d(P_1,K_i)^{-1}n_ik_i^{-1}=\alpha(1-\alpha)^{-1}.$ We have for every graph $H\in\mathcal{B}_0$ that $$h(H,L_i)=(d(H,G_i)n_i+d(H,K_i)k_i)/(d(P_1,G_i)n_i+d(P_1,K_i)k_i)=$$
$$h(H,G_i)(1+d(P_1,K_i)d(P_1,G_i)^{-1}k_in_i^{-1})^{-1}+h(H,K_i)(1+d(P_1,G_i)d(P_1,K_i)^{-1}n_ik_i^{-1})^{-1}.$$ It follows that $$\lim_{i\to\infty} h(H,L_i)=\alpha\lim_{i\to\infty}h(H,G_i)+(1-\alpha)\lim_{i\to\infty}h(H,K_i)$$ holds for every $H\in\mathcal{B}_0$ and thus $\lim_{i\to\infty}\tau(L_i)=\alpha x+(1-\alpha) y$. The compactness of $\mathcal{L}$ follows from lemma \ref{comp}.
\end{proof}

\begin{remark} It follows from lemma \ref{convex} that every finite dimensional projection of the graph log-limit space $\mathcal{L}$ to coordinates given by $H_1,H_2,\dots,H_k\in\mathcal{B}_0$ is a convex compact set. It is not clear whether these convex sets are polytopes i.e. convex hulls of fine point sets. One dimensional projections are closed intervals but the endpoints are not known for every graph $H$. Sidorenko's conjecture says that $h(H,G)\leq |E(H)|$.
\end{remark}

\begin{definition}\label{ergodic} We say that $W\in\mathcal{L}$ is ergodic if $W$ is an extremal point in $\mathcal{L}$.
\end{definition}

Note that according to the Krein-Milman theorem $\mathcal{L}$ is the closed convex hull of ergodic limit objects.
The most natural metric that metrizes log-convergence comes from the definition itself. For two graphs $H_1,H_2\in\mathcal{B}_0$ let us define $$\kappa(G_1,G_2):=\sum_{H\in\mathcal{B}_0}|h(H,G_1)-h(H,G_2)|2^{-|V(H)|^2}.$$
Since there are at most $2^{n^2/2}$ graphs $H$ with $|V(H)|=n$ and $|h(H,G_1)-h(H,G_2)|\leq |V(H)|^2$ we have that the above sum converges. It is clear that convergence in $\kappa$ is equivalent with log-convergence and $\mathcal{L}$ is the completion of $\mathcal{B}_0$ with respect to $\kappa$.

\section{Entropy maximization with marginal constraints}\label{emax}

In this chapter we investigate the following problem. Assume that for a set of random variables $X_1,X_2,\dots X_n$ the joint distributions for certain subsets of the indices $\{1,2,\dots,n\}$ are prescribed. With this constraint what is the maximal possible entropy of the joint distribution of $(X_i)_{i=1}^n$? A trivial example is when the distribution of each individual $X_i$ is given. In this case the entropy is maximized if the random variables are independent. Another example is when the joint distribution of $(X_1,X_2)$ and $(X_2,X_3)$ are given. In this case the two given marginals must have the same marginal on $X_2$ otherwise there is no joint distribution for $(X_i)_{i=1}^3$ satisfying this constraint. If the marginals are given in a consistent way than the so-called conditionally independent coupling of $(X_1,X_2)$ and $(X_2,X_3)$ maximizes the entropy. 

For a precise formulation of the general problem we need some notation.

\begin{definition} Let $H\subseteq 2^V$ be a set system (also called hypergraph) on a finite set $V$. For each $v\in V$ let $F_v$ be a finite set and assume that for each set $S\in H$ there is a probability measure $\mu_S$ on $\prod_{v\in S}F_v$. We denote by $\mathcal{P}(\{\mu_S\}_{S\in H})$ the set of all probability measures $\mu$ on $\prod_{v\in V}F_v$ satisfying $\mu\circ\pi_S^{-1}=\mu_S$ for every $S\in H$ where $\pi_S:\prod_{v\in V}F_v\rightarrow \prod_{v\in S}F_v$ denotes the projection to the coordinates in $S$. We say that the system $\{\mu_S\}_{S\in H}$ is a {\bf consistent system of marginals} if $\mathcal{P}(\{\mu_S\}_{S\in H})$ is not empty. 
\end{definition}

\begin{definition} Let $H\subseteq 2^V$ be a set system and for each $v\in V$ let $F_v$ be a finite set. A probabilty measure $\mu$ on $\prod_{v\in V} F_v$ is called an $H$-Gibbs measure if there are non-negative functions $f_S:\prod_{v\in S}F_v\rightarrow\mathbb{R}\cup\{0\}$  for every $S\in H$ such that $$\mu(x)=z^{-1}\prod_{S\in H}f_S(\pi_S(x))$$ where $z$ is the sum of $\prod_{S\in H}f_S(\pi_S(x))$ over all $x\in\prod_{v\in V}F_v$.
\end{definition}

Using classical tools we get the following proposition. 

\begin{proposition}\label{prop-emax} Assume that $\{\mu_S\}_{S\in H}$ is a consistent system of marginals. Then there is a unique maximizer $\mu$ inside the set $\mu\in\mathcal{P}(\{\mu_S\}_{S\in H})$. Furthermore the measure $\mu$ is an $H$-Gibbs measure.
\end{proposition}

\begin{proof} Using that marginals of convex combinations of measures are the corresponding convex combinations of the marginals we obtain that the set $\mathcal{P}(\{\mu_S\}_{S\in H})$ is a convex set. It is also clear that $\mathcal{P}(\{\mu_S\}_{S\in H})$ is a compact set. The entropy function is a strictly concave continuous function and thus it has a unique maximizer $\mu$ in $\mathcal{P}(\{\mu_S\}_{S\in H})$.
The marginal constraints for $\mu$ can be written in the form of $\mu(\pi_S^{-1}(x))=\mu_S(x)$ where $S$ runs through $H$ and $x$ runs through $\prod_{v\in S}F_v$. These equations are linear equations of the form $\sum_{y\in F}\mu(y)1_x(\pi_S(y))=\mu_S(x)$ for the values of $\mu$ where $F=\prod_{v\in V}F_v$. The principle of maximal entropy says that the entropy maximizer $\mu$ has the form $Z{\rm exp}(\lambda_1f_1+\lambda_2f_2+\dots+\lambda_m f_m)$ for some constants $Z$ and $\{\lambda_i\}_{i=1}^m$ in $\mathbb{R}$ where each $f_i:F\rightarrow\mathbb{R}$ is a function of the form $f_i(y)=1_x(\pi_S(y))$ for some $S\in H$ and $x\in\prod_{v\in S}F_v$. This proves that $\mu$ is an $H$-Gibbs measure.
\end{proof}

In the rest of this chapter we focus on special systems of marginal constraints that are mostly related to our graph limit notion. Roughly speaking we wish to require that in a system of random variables $\{X_v\}_{v\in V}$ indexed by the vertices of a bipartite graph $H$ the marginals $(X_v,X_w)$ are the same distribution $(X_1,X_2)$ for every edge $(v,w)\in E(H)$ with $v\in V_1(H),w\in V_2(H)$. It will turn out that such marginal constraints are always consistent. 

We formulate our definitions in a more general hypergraph setting. Assume that $V=\cup_{i=1}^k V_i$ and that $H\subseteq 2^V$ is such that $|S\cap V_i|=1$ holds for every $S\in H$ and $1\leq i\leq k$. It follows that $|S|=k$ holds for every $S\in H$. In combinatorics $H$ is called a $k$-partite $k$-uniform hypergraph. The set $H$ can also be regarded as a subset in $V_1\times V_2\times\dots\times V_k$. The sepecial case of $k=2$ is the same as our set $\mathcal{B}$ of bipartite graphs with labeled color classes. 

Assume that for every $i$ we associate the same finite set $F_i$ with every element $v\in V_i$.
In other words there is a given bijection $\phi_v:F_v\rightarrow F_i$ for every $1\leq i\leq k$ and $v\in V_i$. For every $S\in H$ there is a bijection $\phi_S:\prod_{v\in S}F_v\rightarrow\prod_{i=1}^kF_i$ given by $\prod_{v\in S}\phi_v$.
Let $\nu$ be a probability measure on $\prod_{i=1}^k F_i$ and let $\mu_S:=\nu\circ\phi_S$ for every $S\in H$.
A convenient fact about the system $\{\mu_S\}_{S\in H}$ is that it is always a consistent system of marginals. This can be seen in the following way. Let $\psi:\prod_{i=1}^k F_i\rightarrow\prod_{v\in V} F_v$ defined by  $$\psi(a_1,a_2,\dots,a_k)=(\phi_v^{-1}(a_i))_{1\leq i\leq k,v\in V_i}.$$
The measure $\mu$ defined by $\mu(T):=\nu(\psi^{-1}(T))$ is in $\mathcal{P}(\{\mu_S\}_{S\in H})$. Assume that the measure $\nu$ is given by the joint distribution $X=\{X_1,X_2,\dots,X_k\}$ where $X_i$ takes values in $F_i$ for $1\leq i\leq k$. Then we denote by $Q(H,X)$ the set $\mathcal{P}(\{\mu_S\}_{S\in H})$. In other words $Q(H,X)$ is the set of all joint distributions $\{X_v\}_{v\in V(H)}$ such that the marginals on the edges of $H$ are all equal to $X$. The consistency of the marginal constraints in this setting justifies the next definition.

\begin{definition}\label{infdens} Let $H$ be a $k$-partite $k$-uniform hypergraph and let $X=(X_1,X_2,\dots, X_k)$ be a joint distribution of $k$ random variables with finite distributions. We denote by $m(H,X)$ the maximal entropy in the set $Q(H,X)$. We introduce the related quantites $$d^*(H,X):=-m(H,X)+\sum_{i=1}^k\mathbb{H}(X_i)|V_i(H)|,$$ $$t^*(H,X):=e^{-d^*(H,X)}$$ and
$$h^*(H,X):=d^*(H,X)/d^*(E_k,X)$$ where $E_k$ denotes the single $k$-edge. (If $X$ is an independent system of random variables then $0=d^*(H,X)=d^*(E_k,X)$. In this case we define $h^*(H,X):=|E(H)|$.)
\end{definition}

Note that $d^*(H,X)$ is the mutual information in the entropy maximizing joint distribution in $Q(H,X)$. In particular $d^*(E_k,X)$ is the mutual information of $(X_1,X_2,\dots,X_k)$. Since mutual information is non-negative it follows that $d^*(H,X)$ is non-negative. 

\begin{remark} If $X=(X_1,X_2,\dots,X_k)$ is not a finite distribution but has finite mutual information (this can be defined through relative entropy) then one can define $d^*(H,X)$ as the infimum of mutual information in the set $Q(H,X)$. 
\end{remark}

 In the next few lemmas we prove various facts about $h^*$ and $d^*$ showing that $d^*(H,X)$ is the analogue of $d(H,G)$, $t^*(H,X)$ is the analogue of $t(H,G)$ and $h^*(H,X)$ is the analogue of $h(H,G)$.
Then we finish the chapter with a theorem that formulates a far reaching connection between $h$ and $h^*$. Let $\mathcal{M}^k$ denote the set of $k$-uniform $k$-partite finite hypergraphs with no isolated points.

\begin{lemma}\label{lem-emax5}  If $H,H'\in\mathcal{M}_k$ are defined on the same vertex set and $E(H')\subseteq E(H)$ then $h^*(H',X)\leq h^*(H,X)$ holds for every finite distribution $X=(X_1,X_2,\dots,X_k)$. 
\end{lemma}

\begin{proof} We have that $Q(H,X)\subseteq Q(H',X)$ and thus $m(H,X)\leq m(H',X)$. Consequently we have $d^*(H,X)\geq d^*(H',X)$ implying $h^*(H,X)\geq h^*(H',X)$.
\end{proof}

\begin{lemma}\label{lem-emax4} Let $X=(X_1,X_2)$ be a finite distribution and assume that $H$ is a tree with at least one edge. Then $h^*(H,X)=|E(H)|$. 
\end{lemma}

\begin{proof} We have by proposition \ref{prop-emax} that the entropy maximizing distribution in $Q(H,X)$ is a Gibbs measure and so it is a Markov random field. This implies that the distribution of every vertex $v$ of degree $1$ is conditionally independent from the remaining vertices with respect to its neighbor. This means that by deleting $v$ the change in $m(H,X)$ is the mutual information $I(X_1;X_2)$. This proves the lemma by induction on the number of edges in $H$. 
\end{proof}

\begin{lemma}\label{lem-emax1} Let $X=(X_1,X_2,\dots,X_k)$ be an arbitrary finite joint distribution and $H\in\mathcal{M}^k$. Then $1\leq h^*(H,X)\leq\prod_{i=1}^k|V_i(H)|$. If $k=2$ then we have the stronger lower bound $$\max(|V_1(H)|,|V_2(H)|)\leq h^*(H,X).$$
\end{lemma}

\begin{proof} We start with the upper bound. By lemma \ref{lem-emax5} it is enough to prove the upper bound for the complete $k$-partite $k$-uniform hypergraph $K$ on the vertex set $\cup_{i=1}^k V_i$.  Observe that the upper bound is equivalent with 
\begin{equation}\label{eq-emax2}
\mathbb{H}(\theta)\geq p\mathbb{H}(X)-\sum_{i=1}^k (p-|V_i|)\mathbb{H}(X_i)
\end{equation}
 where $p=\prod_{i=1}^k |V_i|$ and
$\theta$ is the entropy maximizing distribution in $Q(K,X)$. We go by induction on the number of indices $i$ for which $|V_i|\geq 1$. If $|V_i|=1$ holds for every $1\leq i\leq k$ then the statement is trivial since $h^*(K,\nu)=1$ holds in this case. Assume that the statement holds for some complete $K$ with $|V_i|=1$ for some index $i$. Now we add $r-1$ new vertices to $V_i$ in $K$ and we denote by $K'$ the complete $k$-partite $k$-uniform hypergraph on this vertex set.
Our goal is to construct a probability measure $\theta'$ in $Q(K',X)$ that has high enough entropy to prove the necessary lower bound for the entropy maximizer. Let $\theta'$ denote $r$ fold conditionally independent coupling of $\theta$ with respect to the marginal on $\cup_{j\neq i} V_j$.  It is clear that $\theta'\in Q(K',X)$. Furthermore, following the method in \cite{Sz}, we have that $$\mathbb{H}(\theta')\geq r\mathbb{H}(\theta)-(r-1)\sum_{j\neq i}\mathbb{H}(X_j)|V_j|.$$ Usin (\ref{eq-emax2}) for $\mathbb{H}(\theta)$ in the above inequality we obtain the corresponding version (\ref{eq-emax2}) for $K'$ and thus the induction is complete.  

To prove the lower bound for general $k$ observe that since $H$ has at least one edge and mutual information of random variables is decreasing when taking subsets of variables we get by restricting the entropy maximizing distribution to a single edge that $d^*(H,X)\geq d^*(E_k,X)$.

For the case $k=2$ assume without loss of generality that $|V_1(H)|\geq |V_2(H)|$. Since $H$ has no isolated point there is an edge $e_v$ for every $v\in V_1(H)$. Let $H'$ be the graph whose edge set is $\{e_v|v\in V_1(H)\}$. It is clear that $H'$ is a tree with $|V_1|$ edges. We have by lemma \ref{lem-emax4} that $h^*(H',X)=|V_1|$. Since $h^*(H',X)\leq h^*(H,X)$ the proof is complete.

\end{proof}

\begin{lemma}\label{lem-emax2} Assume that $H\in\mathcal{M}_k$ is the disjoint union of $H_1,H_2\in\mathcal{M}_k$. Let $X=(X_1,X_2,\dots,X_k)$ be a finite joint distribution. Then $m(H,X)=m(H_1,X)+m(H_2,X), d^*(H,X)=d^*(H_1,X)+d^*(H_2,X)$ and $h^*(H,X)=h^*(H_1,X)+h^*(H_2,X)$.
\end{lemma}

\begin{proof} It is clear that the elements of $Q(H,X)$ are all possible couplings of $Q(H_1,X)$ and $Q(H_2,X)$. Thus the entropy maximizer in $Q(H,X)$ is the independent coupling of the entropy maximizers in $Q(H_1,X)$ and $Q(H_2,X)$. This proves the first claim. The remaining two equations are direct consequences of the first one.
\end{proof}

\begin{lemma}\label{lem-emax3} Let $X=(X_1,X_2,\dots,X_k)$ and $Y=(Y_1,Y_2,\dots,Y_k)$ be finite joint distributions and let $X\times Y$ denote the independent coupling $((X_1,Y_1),(X_2,Y_2),\dots,(X_k,Y_k))$. Then for every $H\in\mathcal{M}_k$ we have that $d^*(H,X\times Y)=d^*(H,X)+d^*(H,Y)$.
\end{lemma}

\begin{proof} Assume that $X_i$ is $F_i$-valued and $Y_i$ is $L_i$-valued for $1\leq i\leq k$. Let $P_X=\prod_{i=1}^k F_i^{V_i(H)}$ and $P_Y=\prod_{i=1}^k L_i^{V_i(H)}$. Let $\nu_X$ (resp. $\nu_Y$) denote the probability measure on $P_X$ (resp. $P_Y$) representing $X$ (resp. $Y$). We have that $X\times Y$ is represented by $\nu_X\times\nu_Y$ on $P_X\times P_Y$.  If $\mu\in Q(H,X\times Y)$ then let $\mu_X$ denote the marginal of $\mu$ on $P_X$ and let $\mu_Y$ denote marginal of $\mu$ on $P_Y$. We have that $\mu_X\times\mu_Y\in Q(H,X\times Y)$ and that $\mathbb{H}(\mu_X\times\mu_Y)\geq\mathbb{H}(\mu)$. It follows that the entropy maximizer in $Q(H,X\times Y)$ is the product of the entropy maximizers in $Q(H,X)$ and $Q(H,Y)$.
\end{proof}

A novelty of definition \ref{infdens} is that it gives a natural definition for sugbraph densities in joint distributions of random variables. We believe that the quantities $m(H,X),d^*(H,X)$ and $h^*(H,X)$ are useful information theoretic invariants of joint distributions. The relationship between the quantities $h^*(H,X)$ and $h(H,G)$ is explained by the next theorem. 

\begin{theorem}\label{main} For every finite joint distribution $X=(X_1,X_2)$ there is a log-convergent graph sequence $\{G_i\}_{i=1}^\infty$ such that $G_i$ is edge-vertex transitive for every $i\in\mathbb{N}$ and $$\lim_{i\to\infty}h(H,G_i)=h^*(H,X)$$ holds for every $H\in\mathcal{B}_0$.
\end{theorem}

\begin{proof} We assume that $X_1$ is a probability distribution on $F_1$ and $X_2$ is a probability distribution on $F_2$. Thus $X$ is represented by a probability measure  $\nu$ on $F_1\times F_2$. We denote the distributions of $X_1$ and $X_2$ by $\nu_1$ and $\nu_2$.  Note first that $Q(H,X)$ depends continuously on $\nu$ and thus $m(H,X)$ and $h^*(H,X)$ are also continuous in $\nu$. Consequently it is enough to prove the statement for the case where all probabilities in $\nu$ are rational numbers. This implies in particular that both marginals are given by rational probabilities. 

In this proof we will use the convention that if $e$ is an element in some product set $F^n$ then we denote by ${\rm distr}(e)$ the probability distribution on $F$ obtained by choosing a uniformly random coordinate of $e$. It is clear that exactly those probability distributions can be produced this way for a fix $n$ where each probability is of the form $a/n$ for some integer $a$. The symmetric group $S_n$ acts on $F^n$ by permuting the coordinates. It is clear that $e_1,e_2\in F^n$ are in the same orbit of $S_n$ if and only if ${\rm distr}(e_1)={\rm distr}(e_2)$.

We denote by $V_{1,n}$ (resp. $V_{2,n}$) the subset of elements $v$ in $F_1^{n!}$ (resp. in $F_2^{n!}$) in which ${\rm distr}(v)=\nu_1$ (resp. ${\rm distr}(v)=\nu_2$). If $n$ is big enough then $V_{1,n}$ and $V_{2,n}$ are non empty using the rationality of the probabilities. Viewing $V_{1,n}\times V_{2,n}$ as a subset in $(F_1\times F_2)^{n!}$ we denote by $E_n$ the set of elements $e$ in $V_{1,n}\times V_{2,n}$ that satisfy ${\rm distr}(e)=\nu$. Again if $n$ is big enough then $E_n$ is non empty. The triple $G_n:=(V_{1,n},V_{2,n},E_n)$ is a bipartite graph such that the symmetric group $S_{n!}$ acts on it by permuting the coordinates. Since $E_n$ is given by a fix distribution it follows that $S_{n!}$ acts transitively on $E_n$ and thus $G_n$ is edge transitive. Note that $G_n$ is embedded into $K^{n!}$ as an $S_n$ invariant sub-graph where $K$ is the complete graph with $V_1(K)=F_1,V_2(K)=F_2,E(K)=F_1\times F_2$.

Let $H\in\mathcal{B}_0$ be some fixed graph. The group $S_{n!}$ acts on the homomorphism set $\Hom(H,G_n)$ by $(f^\pi)(x)=f(x)^\pi$ where $\pi\in S_n,x\in V(H)$ and $f\in\Hom(H,G_n)$. The fact that $S_n$ acts as automorphisms on $G_n$ guarantees that images of homomorphisms are homomorphisms.
The key idea of the proof is that the number of orbits of $S_n$ on $\Hom(H,G_n)$ is polynomial in $n!$ however the size of the largest orbit is exponential. Thus the size of the largest orbit dominates the logarithm of $|\Hom(H,G_n)|$ when normalized by $n!$. We need the next claim.

\noindent{\it~Claim:  ~ Let $a_n$ denote the size of the largest orbit of $S_{n!}$ on $\Hom(H,G_n)$. Then $\lim_{i\to\infty}\log(a_n)/n!=m(H,X)$}

Let $O$ be an orbit of $S_{n!}$ on $\Hom(H,G_n)$. Assume that $f\in O$ is some element. Since $G_n$ is embedded into $K^{n!}$ we have that $f\in\Hom(H,K^{n!})$ and thus $f$ can be represented as a sequence $\{f_i\}_{i=1}^{n!}$ where each $f_i$ is an element in $\Hom(H,K)$. Let $\mu={\rm distr}(f)$. We have that $O=\{g|g\in\Hom(H,K^{n!}),{\rm distr}(g)=\mu)$. It follows by basic properties of entropy that $|\log(|O|)/n!-\mathbb{H}(\mu)|=o(1)$ uniformly for every orbit $O$ if $n$ is large enough. Observe that $\mu$ is a probability distribution on $F_1^{V_1(H)}\times F_2^{V_2(H)}$ with the property that the marginal on every edge of $H$ is equal to $\nu$. This is clear from the fact that these marginals represent edges in $G_n$ because $f$ is a homomorphism. We obtain that $\log(|O|)/n!\leq m(H,X)+o(1)$. To finish the proof of the claim we need to find an orbit $O$ with $\log(|O|)=m(H,X)+o(1)$.
The idea is to discretize the probability distribution $\theta$ in $Q(H,X)$ that maximizes entropy. If we manage to find $\theta'$ in $Q(H,X)$ with the property that $d_{TV}(\theta,\theta')=o(1)$ for the total variation distance $d_{TV}$ and $\theta'(x)n!\in\mathbb{Z}$ for every elementary event $x$ then $\theta'$ represents an orbit of homomorphisms of $H$ into $G_n$ with the desired property. The set $Q(H,X)$ is a convex set defined by rational inequalities. It follows that extremal points of $Q(H,X)$ have rational coordinates and thus rational points are dense in $Q(H,X)$. We obtain that $\theta$ can be approximated arbitrarily well by rational probability distributions inside $Q(H,X)$. If $n$ is large enough then any such approximation $\theta'$ will have the integrality property $\theta'(x)n!\in\mathbb{Z}$. The proof of the claim is thus finished. 

Let $b_n$ denote the number of orbits of $S_{n!}$ on $\Hom(H,G_n)$. Each orbit is represented by a probability distribution on $F_1^{V_1(H)}\times F_2^{V_2(H)}$ with the property that elementary events have probabilitis of the form $r/n!$ for some integer $0\leq r\leq n!$. This means that $b_n\leq (n!+1)^t$ where $t=|F_1|^{V_1(H)}|F_2|^{V_2(H)}$. Now we have that $a_n\leq\Hom(H,G_n)\leq a_nb_n$ and thus
$$\log(a_n)/n!\leq \log(|\Hom(H,G_n)|)/n!\leq \log(a_n)/n!+\log(b_n)/n!.$$
We have by our estimate that $\log(b_n)/n!=o(1)$ and thus 
\begin{equation}\label{eq-emax1}
\log(|\Hom(H,G_n)|)/n!=m(H,X)+o(1).
\end{equation}

Observe that $\log(|V_{i,n}|)/n!=\mathbb{H}(\nu_i)+o(1)$ for $i=1,2$. Thus we have by (\ref{eq-emax1}) that
$$\log(t(H,G_n))/n!=m(H,X)-|V_1(H)|\mathbb{H}(\nu_1)-|V_2(H)|\mathbb{H}(\nu_2)+o(1)=d^*(H,X)+o(1).$$
Using the above equation we obtain that $h(H,G_n)=h^*(H,X)+o(1)$ finishing the proof.

\end{proof}

\section{An information theoretic limit concept}

The goal of this chapter is to introduce limit concepts for joint distributions of $k$ random variables where $k$ is fixed. In chapter \ref{emax} we have introduced various ways of testing a joint distribution $X=(X_1,X_2,\dots,X_k)$ by a finite $k$-partite $k$-unifrom hypergraph. These can be used to introduce limit concepts in information theory. The limit concept related to $d^*$ (or equivalently to $t^*$) is very similar to dense graph and hypergraph convergence. In this paper we are interested in convergence corresponding to the quantities $h^*$ and especially in the case $k=2$. 

\begin{definition} Let $\{X^i=(X_1^k,X_2^i,\dots,X_k^i)\}_{i=1}^\infty$ be a sequence of finite joint distributions. We say that $\{X^i\}_{i=1}^\infty$ is $h^*$-convergent (resp. $d^*$-convergent) if we have that $\lim_{i\to\infty} h^*(H,X^i)$ (resp. $\lim_{i\to\infty} d^*(H,X^i)$) exists for every $H\in\mathcal{M}^k$.
\end{definition}

Lemma \ref{lem-emax1} implies the convenient property of $h^*$ convergence that every sequence of joint distributions of $k$ random variables has a $h^*$-convergent subsequence. Similarly to the graph log-limit space $\mathcal{L}$ we denote by $\mathcal{L}^*_k$ the limit space of $k$-fold joint distributions in $\mathbb{R}^{\mathcal{M}_k}$. A function $f:\mathcal{M}_k\rightarrow\mathbb{R}$ is in $\mathcal{L}^*_k$ if and only if there is a sequence $\{X^i\}_{i=1}^\infty$ of $k$-fold joint distributions such that $f(H)=\lim_{i\to\infty}h^*(H,X^i)$ holds for every $H\in\mathcal{M}_k$.
It follows from lemma \ref{lem-emax2} following the same argument as in lemma \ref{convex} that $\mathcal{L}^*_k$ is a convex compact set. Similarly to definition \ref{ergodic} we say that $W\in\mathcal{L}^*$ is ergodic if it is an extreme point. If $k=2$ we use the short-hand notation $\mathcal{L}^*$ for $\mathcal{L}^*_2$.
An immediate corollary of theorem \ref{main} is that $\mathcal{L}^*$ is contained in $\mathcal{L}$.

\begin{definition} For a graph $G\in\mathcal{B}_0$ let $X_G=(X_1,X_2)$ denote the distribution of a uniform random edge in $G$ where $X_1\in V_1(G)$ and $X_2\in V_2(G)$ are the endpoints of the edge. By abusing the notation we introduce $d^*(H,G):=d^*(H,X_G),~t^*(H,G):=t(H,X_G)$ and $h^*(H,G):=h^*(H,X_G)$.
\end{definition}

\begin{lemma}\label{lem-inflim1} Let $G\in\mathcal{B}_0$. Then $h^*(H,G)\geq h(H,G)$ holds for every $H\in\mathcal{B}_0$. Furtehrmore if $G$ is edge-vertex transitive then $h^*(H,G)=h(H,G), d^*(H,G)=d(H,G)$ and $t^*(H,G)=t(H,G)$ holds for every $H\in\mathcal{B}_0$.
\end{lemma}

\begin{proof} We start with a few observations. It is clear that $\log(|V_i(G)|)\geq\mathbb{H}(X_i)$ for $i=1,2$ since $\log|V(G_i)|$ is the entropy of the uniform distribution on $V(G_i)$ and uniform distribution has the maximal entropy. Similarly $\log(|\Hom(H,G)|)\geq m(H,X_G)$ holds since every distribution in $Q(H,X_G)$ is concentrated on the homomorphism set $\Hom(H,G)$. Observe that we have by definition that $\mathbb{H}(X_G)=\log(|E(G)|)$. From the definition of $h^*(H,X_G)$ we have
$$m(H,X)=h^*(H,X_G)\mathbb{H}(X_G)-\sum_{i=1}^2(h^*(H,X_G)-|V_i(H)|)\mathbb{H}(X_i)$$
and thus by the previous observations and lemma \ref{lem-emax1} we obtain
$$\log(|\Hom(H,G)|)\geq h^*(H,X_G)\log(|E(G)|)-\sum_{i=1}^2(h^*(H,X_G)-|V_i(H)|)\log(|V_i(G)|)$$
which is equivalent with the first statement. 

To see the second statement we have to check that all the inequalities used above become equalities and that $d^*(P_1,G)=d(P_1,G)$.
The fact that $G$ is edge-vertex transitive implies that the automorphism group of $G$ acts transitively on both $V_1(G)$ and $V_2(G)$ and thus the marginals of $X_1$ and $X_2$ of $X$ are uniform. It follows that $\log(|V_i(G)|)=\mathbb{H}(X_i)$ for $i=1,2$. It follows that $d^*(P_1,G)=d(P_1,G)$. Edge-vertex transitivity implies that the uniform measure $\mu$ on $\Hom(H,G)$ has uniform marginals on the edges and thus $\mu\in Q(H,X_G)$. It follows that $\log(|\Hom(H,G)|)\leq m(H,X_G)$ and this together with the opposite inequality from above implies $\log(|\Hom(H,G)|)=m(H,X_G)$.
\end{proof}

From lemma \ref{lem-inflim1} and theorem \ref{main} we obtain the following group theoretic characterization of the information theoretic limit space $\mathcal{L}^*$.

\begin{theorem}\label{thm-inflim1} The closure of all edge-vertex transitive graphs with respect to log-convergence (represented in $\mathcal{L}$) is equal to $\mathcal{L}^*$.
\end{theorem}

This is a somewhat surprising connection between information theory and group theory.
We finish with a set of linear equations that $\mathcal{L}^*$ satisfies within $\mathcal{L}$

\begin{lemma}\label{lem-inflim2} Let $W\in\mathcal{L}^*$. Then we have the following two properties
\begin{enumerate}
\item $h(H,W)=h(H_1,W)+h(H_2,W)$ if $H$ is obtained from $H_1$ and $H_2$ by identifying a vertex.
\item $h(H,W)=h(H_1,W)+h(H_2,W)-1$ if $H$ is obtained from $H_1$ and $H_2$ by identifying an edge.
\end{enumerate}
\end{lemma}

\begin{proof} We have that $W$ is a limit of edge-vertex transitive graphs so it is enough to prove it in the case when $W$ is such a graph. The first equation follows from vertex transitivity since every vertex of $W$ has the same number of copies of $H_1$ and $H_2$ and thus $t(H,W)=t(H_1,W)t(H_2,W)$. The second statement follows in a similar way from edge transitivity.
\end{proof}

\begin{question} Is $\mathcal{L}^*$ characterized by $\mathcal{L}^*\subset\mathcal{L}$ and the linear equations in lemma \ref{lem-inflim2}?
\end{question}

\section{Sparsity exponent}

In dense graph limit theory sparsity (or density) is described by the edge density $t(P_1,G)$.
The natural analogue of edge density in the logarithmic framework is the power $\beta$  to which we have to raise the number of the edges in the complete graph on $V(G)=V_1(G)\cup V_2(G)$ (which is equal to $|V_1(G)||V_2(G)|$) to obtain the number of edges in $G$. Unfortunately this sparsity exponent can not be read off in a simple way using the parameters $h(H,G)$. (Note that $h(P_1,G)$ is always $1$ so it gives no information.)
In this chapter we show a connection between the asymptotic behavior of the graph parameter $H\mapsto h(H,G)$ and the sparsity exponent. We also study how to extend the notion of sparsity to the log limit space $\mathcal{L}$. 

 Let $G\in\mathcal{B}_0$ be a graph, let $$\beta_v(G):=\log|E(G)|/(\log|V_1(G)|+\log|V_2(G)|)$$ and let 
$$\beta_e(G):=\mathbb{H}(X_G)/(\mathbb{H}(X_1)+\mathbb{H}(X_2))$$
where $X_G=(X_1,X_2)$ is a uniform random edge in $G$ with endpoints $X_1$ and $X_2$. 
Using that $\mathbb{H}(X_G)=\log |E(G)|$ , $\log |V_i|\geq \mathbb{H}(X_i)$ for $i=1,2$ and that $0\geq I(X_1;X_2)=\mathbb{H}(X_1)+\mathbb{H}(X_2)-\mathbb{H}(X_G)$ we have that $0<\beta_v(G)\leq\beta_e(G)\leq 1$. If $G$ is regular (i.e. there are two numbers $a,b$ such that every vertex in $V_1$ has degree $a$ and every vertex in $V_2$ has degree $b$) then $X_1$ and $X_2$ have uniform distributions and thus $\beta_v(G)=\beta_e(G)$.  Intuitively we can view $\beta_e(G)$ as an ``edge version'' of sparsity where vertices of small degree count less. If we add isolated points to $G$ then $\beta_e(G)$ does not change. Note that the quantity $\beta_e$ can naturally be extended to arbitrary finite joint distributions $X=(X_1,X_2)$ by essentially the same formula.

It is clear that $\beta_v(G)$ and $\beta_e(G)$ are not determined by $\tau(G)\in\mathcal{L}$ since $h(H,G)=h(H,G_m)$ holds if $G_m$ is an $m$-fold blow up of $G$ however if $m$ goes to infinity we have that $\lim_{m\to\infty}\beta_v(G_m)=\lim_{m\to\infty}\beta_e(G_m)=1$. Despite of this fact it will turn out that if $G$ is regular and twin free (i.e. there are no two distinct vertices with identical neighborhood) then we can reconstruct $\beta_v(G)=\beta_e(G)$ from $\tau(G)$. We continue with two sparsity notions on the log-limit space $\mathcal{L}$.

\begin{definition} For $W\in\mathcal{L}$ let $\beta_0(W)$ denote the infimum of the numbers $\alpha$ such that there is a log-convergent graph sequence $\{G_i\}_{i=1}^\infty$ with limit $W$ and $\liminf_{i\to\infty}\beta_v(G_i)=\alpha$. Let furthermore $\hat{\beta}(W):=\sup_{n\in\mathbb{N}}~(1-1/g_n(W))$ where
$$g_n(W):=h(K_{2,n},W)+h(K_{n,2},W)-h(K_{1,n},W)-h(K_{n,1},W)$$ and $K_{a,b}$ is the complete bipartite graph with $|V_1(K_{a,b})|=a$ and $|V_2(K_{a,b})|=b$.

\end{definition}

\begin{proposition} The parameters $\beta_0,\hat{\beta},\beta_v$ and $\beta_e$ have the following properties.
\begin{enumerate}
\item $\hat{\beta}$ and $\beta_0$ are lower semi continuous i.e. if $\{W_i\}_{i\to\infty}$ is a convergent sequence in $\mathcal{L}$ with limit $W$ then $\liminf_{i\to\infty}\hat{\beta}(W_i)\geq\hat{\beta}(W)$ and $\liminf_{i\to\infty}\beta_0(W_i)\geq\beta_0(W)$.
\item If $W\in\mathcal{L}$ then $\hat{\beta}(W)\leq\beta_0(W)$.
\item If $G\in\mathcal{B}_0$ is arbitrary then $\hat{\beta}(G)\leq\beta_0(G)\leq\beta_v(G)\leq\beta_e(G)$.
\item If $G\in\mathcal{B}_0$ is a regular twin free graph then $\hat{\beta}(G)=\beta_0(G)=\beta_v(G)=\beta_e(G)$. 

\end{enumerate}
\end{proposition}

\begin{proof} We start with the first statement. Assume that $\{W_i\}_{i=1}^\infty$ converges to $W$ in $\mathcal{L}$. By definition we have $\lim_{i\to\infty} g_n(W_i)=g_n(W)$ for every $n$ and thus $\liminf_{i\to\infty}\hat{\beta}(W_i)\geq g_n(W)$. This implies the lower semicontinuity of $\hat{\beta}$.

To see the lower semicontinuity of $\beta_0$ choose elements $G_i\in\mathcal{B}_0$ such that $\kappa(G_i,W_i)\leq 1/n$ and $|\beta_v(G_i)-\beta_0(W_i)|\leq 1/n$. We have that $\liminf_{i\to\infty} \beta_v(G_i)=\liminf_{i\to\infty} \beta_0(W_i)$ and that $\lim_{i\to\infty} G_i=W$. This shows that $\beta_0(W)\leq\liminf_{i\to\infty}\beta_0(W_i)$.

We continue with the proof of $\hat{\beta}(G)\leq\beta_v(G)$ for $G\in\mathcal{B}_0$. For $v,w\in V_i(G)$ let $A_{v,w,i}$ denote the number of common neighbors of $v$ and $w$ in $G$. Let $$T_n:=\sum_{i=1}^2\Bigl(\log\Bigl(\sum_{v,w\in V_i(G)}A_{v,w,i}^n\Bigr)-\log\Bigl(\sum_{v\in V_i(G)}A_{v,v,i}^n\Bigr)\Bigr).$$
Note that the four terms in the above sum are the logarithms of $|\Hom(K_{2,n},G)|$ and $|\Hom(K_{n,2},G)|$ with plus sign and the logarithms of $|\Hom(K_{1,n},G)|$ and $|\Hom(K_{n,1},G)|$ with minus sign.
Using this fact an elementary calculation shows that
\begin{equation}\label{usq-eq1}
g_n(G)=(\log|V_1(G)|+\log|V_2(G)|-T_n)/(\log |V_1(G)|+\log |V_2(G)|-\log |E(G)|).
\end{equation}
Observe that by $t(K_{1,n},G)\geq t(K_{2,n},G)$ and $t(K_{n,1},G)\geq t(K_{n,2},G)$ we have that $g_n(G)\geq 0$. 
Thus by $T_n\geq 0$ and (\ref{usq-eq1}) we obtain that $\beta_v(G)\geq 1-1/g_n(G)$. This proves that $\hat{\beta}(G)\leq\beta_v(G)$.

We prove now that $\hat{\beta}(W)\leq\beta_0(W)$ holds for $W\in\mathcal{L}$. It is clear that we can choose a sequence $\{G_i\}_{i=1}^\infty$ in $\mathcal{B}_0$ with limit $W$ such that $\lim_{i\to\infty}\beta_v(G_i)=\beta_0(W)$. Using the lower semicontinuity of $\hat{\beta}$ and the fact that $\hat{\beta}(G_i)\leq\beta_v(G_i)$ we obtain that $\hat{\beta}(W)\leq\liminf_{i\to\infty}\hat{\beta}(G_i)\leq\beta_0(W)$.

Now let us assume that $G$ is twin free and regular. To show $\beta_v(G)=\hat{\beta}(G)$ it is enough to prove that $\lim_{n\to\infty} T_n=0$. This is easy to see from the fact that $A_{v,v,i}=d_i$ holds universally in $V_i$ where $d_1$ and $d_2$ are the uniform degrees and furthermore $A_{v,w,i}<d_i$ holds if $v\neq w$ are in $V_i$.  

To complete the proof we need to show that $\beta_0(G)\leq\beta_v(G)$ holds for $G\in\mathcal{B}_0$. This is trivial since the constant sequence $G$ converges to $G$ in $\mathcal{L}$.
\end{proof}


\section{Quasi-randomness}\label{chaprand}

In dense graph limit theory a sequence of graphs $\{G_i\}_{i=1}^\infty$ is quasi random with density $0\leq p\leq 1$ if $\lim_{i\to\infty}t(H,G_i)=p^{|E(H)|}$ holds for every graph $H\in\mathcal{B}$. For $0<p\leq 1$ these sequences are log-convergent but their limit in $\mathcal{L}$ does not depend on $p$. The limit object is always the graph parameter defined by $f(H):=|E(H)|$. In other words there is a unique dense random object (represented by $f$) in the graph log-limit space. However we show in this chapter that log-convergence differentiates between an infinite family of different sparse quasi-random objects related to sparsity exponents. 

For fix $0<\beta\leq 1$ and $0<\alpha<1$ let $G=G(n,\beta,\alpha)$ denote the random graph model where we have that $|V_1(G)|=\lceil n^{\alpha}\rceil, |V_2(G)|=\lceil n^{1-\alpha}\rceil$ and edges are created between pairs of vertices $v\in V_1(G), w\in V_2(G)$ independently with probability $n^{\beta-1}$. We investigate the log-limits of such random graphs where $\beta,\alpha$ are fixed and $n$ goes to infinity.

\begin{definition} For a graph $H\in\mathcal{B}_0$, $0<\beta\leq 1$ and $0<\alpha<1$  let $\alpha_1:=\alpha,\alpha_2:=1-\alpha$ and let $R(\beta,\alpha,H)$ denote the minimum of
\begin{equation}\label{quasi-eq1}
|E(H')|+(1-\beta)^{-1}\sum_{i=1}^2  (|V_i(H)|-|V_i(H')|)\alpha_i 
\end{equation}
where $H'$ runs through all homomorphic images of $H$ (this means that there is a homomorphism from $H$ to $H'$ which is surjective on the vertices and on the edges of $H$.) We denote by $R(\beta,\alpha)$ the graph parameter that maps $H$ to $R(\beta,\alpha,H)$.
\end{definition} 

Note that if $\beta=1$ then it is natural to define $R(\beta,\alpha,H)$ to be $|E(H)|$ since this is the limit of it as $\beta$ goes to $1$. In general we have that $0\leq R(\beta,\alpha,H)\leq |E(H)|$ where the uppur bound is given by the choice $H'=H$. The next proposition implies that $R(\beta,\alpha)$ is a graph parameter in $\mathcal{L}$ and that it can be obtained as the limit of Erd\H os-R\'enyi type random graphs. 
In the rest of this chapter we prove the next theorem.

\begin{theorem}\label{quasi-thm} For every fix pair $0<\beta\leq 1, 0<\alpha<1$ and graph $H\in\mathcal{B}_0$ we have that $h(H,G(n,\beta,\alpha))$ converges to $R(\beta,\alpha,H)$ in probability as $n$ goes to infinity. It implies that $R(\beta,\alpha)\in\mathcal{L}$.
\end{theorem} 

Note that the notion of convergence in probability makes sense if random variables take values in $\mathbb{R}\cup\{\infty\}$ where the $\infty$ symbol stands for ``not defined''. This extension is important since with a very small probability $G=G(n,\beta,\alpha)$ is empty and thus $h(H,G)$ is not defined in this case. 
To prove theorem \ref{quasi-thm} we will need some preparation. For maintaining symmetry in our formulas let us introduce $\alpha_1:=\alpha$ and $\alpha_2:=1-\alpha$.
For a graph $H\in\mathcal{B}_0$ let $D(H):=\alpha_1|V_1(H)|+\alpha_2|V_2(H)|-(1-\beta)|E(H)|$ and let $M(H)$ denote the minimum of $D(H')$ where $H'$ runs through the subgraphs in $H$. Note that the quantities $D(H)$ and $M(H)$ depend on $\alpha_1,\alpha_2,\beta$ but these constants are fixed throughout the proof of theorem \ref{quasi-thm}. We will use the short hand notation $G_n$ for the random graph model $G(n,\beta,\alpha)$.
For two graph $H$ and $G$ let us denote by $\Hom_0(H,G)$ the set of injective homomorphisms from $H$ to $G$. 
We will use the next logarithmic version of Chebyshev's inequality.

\begin{lemma}\label{logcheb} Let $\{X_i\}_{i=1}^\infty$ be a sequence of non negative random variables. Assume that  $\lim_{i\to\infty}\mathbb{E}(X_i)=+\infty$ and that $\lim_{i\to\infty}\sigma(X_i)/\mathbb{E}(X_i)=0$. Then $\log X_i/\log \mathbb{E}(X_i)$ converges to $1$ in probability as $i$ goes to infinity.
\end{lemma}

\begin{proof} We have that $\mathbb{P}(|\log X_i/\log \mathbb{E}(X_i)-1|\geq\epsilon)$ is equal to $$\mathbb{P}(X_i\geq\mathbb{E}(X_i)^{1+\epsilon})+\mathbb{P}(X_i\leq \mathbb{E}(X_i)^{1-\epsilon})$$ which is less than
$$2\mathbb{P}(|X_i-\mathbb{E}(X_i)|\geq \mathbb{E}(X_i)-\mathbb{E}(X_i)^{1-\epsilon})$$ if $\mathbb{E}(X_i)$ is big enough. Using that $\sigma(X_i)=o(\mathbb{E}(X_i))$ and $\mathbb{E}(X_i)^{1-\epsilon}=o(\mathbb{E}(X_i))$ we obtain by Chebyshev's inequality that the above probability goes to $0$.
\end{proof}

\medskip

The next lemma is a basically a bipartite version of a result by Bollob\'as \cite{B}.

\begin{lemma}\label{quasi-lem2} Let $H\in\mathcal{B}_0$ such that $M(H)>0$. Then $\log |{\rm Hom}_0(H,G_n)|/\log n$ converges to  $D(H)$ in probability as $n$ goes to infinity.
\end{lemma}

\begin{proof} Let $X_n$ be the random variable $|\Hom_0 (H,G_n)|$. We start by computing $\mathbb{E}(X_n)$. Let $L_n$ be the set of pairs of injective maps $V_1(H)\rightarrow V_1(G_n),V_2(H)\rightarrow V_2(G_n)$. We have that $|L_n|=n^{\alpha_1|V_1(H)|+\alpha_2|V_2(H)|+o(1)}$. For every $\phi\in L_n$ the probability that $\phi$ gives a homorphism is $n^{(\beta-1)|E(H)|}$. Thus $\mathbb{E}(X_n)$ is $n^{D(H)+o(1)}$. Using lemma \ref{logcheb} it is enough to show that $\sigma(X_n)=o(\mathbb{E}(X_n))$ so we continue by estimating the variance of $X_n$.

Each element $\phi\in L_n$ gives a copy $\phi(H)$ of $H$ on the vertex set $V_1(G_n)\cup V_2(G_n)$. For $\phi\in L_n$ let $I_\phi$ be the indicator function of the event that $\phi(H)\subseteq G_n$. We write $\phi\sim\psi$ if $E(\phi(H))\cap E(\psi(H))\neq\emptyset$. 
We have that $${\rm Var}(X_n)=\sum_{\phi\in L_n}\sum_{\psi\sim\phi} cov(I_\psi,I_\phi)=\sum_{\phi\in L_n}\sum_{\psi\sim\phi} \mathbb{E}(I_\psi I_\phi)\leq O\bigl(\sum_{H'\subseteq H}n^{2D(H)-D(H')}\Bigr).$$
It follows that
$${\rm Var}(X_n)/\mathbb{E}(X_n)^2=O\Bigl(\sum_{H'\subseteq H}n^{-D(H')}\Bigr)=o(1).$$ This completes the proof.

\end{proof}

\medskip

\begin{lemma}\label{quasi-lem1} Let $H'$ be a homomorphic image of a graph $H\in\mathcal{B}_0$ that maximizes $D(H')$. Then $M(H')>0$.
\end{lemma}

\begin{proof} Assume by contradiction that $H_2$ is a subgraph in $H'$ with $D(H_2)\leq 0$. Let $H_3$ be the graph obtaind from $H'$ be contracting $V_1(H_2)\subseteq V_1(H')$ and $V_2(H_2)\subseteq V_2(H')$ to a single point $v_1$ and $v_2$ and then reducing multiple edges. Observe that $|E(H_2)|\geq 1$. It is clear that $H_3$ is a homomorphic image of $H'$ in which $v_1$ and $v_2$ are connected. Using $|V_i(H_3)|=|V_i(H')|-|V_i(H_2)|+1$ we have that $$D(H_3)-D(H')=(1-\beta)(|E(H')|-|E(H_3)|)+\sum_{i=1}^2\alpha_i(1-|V_i(H_2)|)=$$
$$(1-\beta)(|E(H')-|E(H_3)|-|E(H_2)|)+1-D(H_2).$$ Since $(1-\beta)$ and $-D(H_2)$ are non-negative it is enough to show that $|E(H')|+1/(1-\beta)\geq |E(H_3)|+E(H_2)|$. Let $\phi\in\Hom(H',H_3)$ be the homomorphism constructed above. We have that $|E(H')|=\sum_{e\in E(H_3)}|\phi^{-1}(e)|$. notice that for $e=(v_1,v_2)$ we have that  $|\phi^{-1}(e)|=|E(H_2)|$ and thus $|E(H')|\geq |E(H_3)|+|E(H_2)|-1$. Using that $1/(1-\beta)>1$ the proof is complete.
\end{proof}

\medskip

\noindent{\it Proof of theorem \ref{quasi-thm}.}~ Let us define the random variables $X_n:=|\Hom(H,G_n)|, Y_n:=|\Hom(P_1,G_n)|$ . We have that $$h(H,G_n)=(\alpha_1|V_1(H)|+\alpha_2|V_2(H)|-\log X_n/\log n)(1-\log Y_n/\log n)+o(1).$$ where the error $o(1)$ comes from the rounding error between $n^{\alpha_i}$ and $\lceil n^\alpha_i\rceil$. It remains to prove that $\log Y_n/\log_n$ converges to $\beta$ and $\log X_n/\log n$ converges to 
\begin{equation}\label{quasi-eq2}
\alpha_1|V_1(H)|+\alpha_2|V_2(H)|-(1-\beta)R(\beta,\alpha,H)
\end{equation}
 in probability. The first statement follows (by using lemma \ref{logcheb}) from the fact that $Y_n$ is the sum of $n+o(n)$ independent random variables that are the characteristic functions of the edges in $G_n$ and so $\mathbb{E}(Y_n)=n^{\beta}(1+o(1))$ and $\sigma(Y_n)=n^{\beta/2}(1+o(1))$.
 
 Observe that (\ref{quasi-eq2}) is equal to the maximum $D$ of $D(H')$ where $H'$ runs through the homomorphic images of $H$. Let us choose a maximizer $H'$. By lemma \ref{quasi-lem1} we have that $M(H')>0$. Thus by lemma \ref{quasi-lem2} we obtain that $\log |{\rm Hom}_0(H',G_n)|/\log n$ converges to $D$ in probability. Using that $|{\rm Hom}_0(H',G_n)|\leq|\Hom(H,G_n)|$ we obtain that $\mathbb{P}(\log X_n/\log n\leq D-\epsilon)=o(1)$ for every $\epsilon>0$.  To prove the upper bound notice that $|\Hom(H,G_n)|=\sum_K |\Hom_0(K,G_n)|$ where $K$ runs through the homomorphic images of $H$. Note that fro each fix homomorphic image $K$ we have that $\mathbb{E}(|\Hom_0(K,G_n)|)=n^{D(K)+o(1)}$ (see the proof of lemma \ref{quasi-lem2}). This implies that $\mathbb{E}(|\Hom(H,G_n)|)=O(n^{D(K)+o(1)})$. This implies by Markov's inequality that $\mathbb{P}(\log X_n/\log n\geq D+\epsilon=o(1)$.
 
\medskip

\begin{question} In general we have in $\mathcal{L}$ that $h(C_4,W)\leq 4$. In the spirit of the famous Chung-Graham-Wilson theorem \cite{CGW} it is interesting to study what happens at the extremal value $h(C_4,W)=4$. It is easy to see that $h(C_4,R(\beta,1/2))=4$ for every $3/4\leq\beta\leq 1$. Is it true that $h(C_4,W)=4$ implies that $W$ is a convex combination of quasi-random elements $R(\beta,\alpha)$ in $\mathcal{L}$?
\end{question}

The next question is related to Sidorenko's conjecture:

\begin{question} Is $R(\beta,\alpha)$ an ergodic element (extreme point) in $\mathcal{L}$?
\end{question}

\section{Applications}

Our results on log-convergence and $h^*$-convergence create an interesting link between graph theory, information theory and group theory. We demonstrate this link by some applications.

For a bipartite graph $H$ let $c(H)$ be the smallest real number such that $t(H,G)\geq t(P_1,G)^{c(H)}$ holds for every graph $G\in\mathcal{B}$. A famous conjecture of Sidorenko says that $c(H)=|E(H)|$ holds for every bipartite graph and it is checked for various families of graphs.  Independently from the fact whether Sidorenko's conjecture is true or false in general it is an important problem in extremal combinatorics to determine $c(H)$ for every bipartite graph. It is clear that using our notation $c(H)=\sup_{G\in\mathcal{B}_0} h(H,G)$. The next theorem gives an information theoretic and a group theoretic characterization for $c(H)$.

\begin{theorem}\label{thm-app1} We have for an arbitrary bipartite graph $H$ (with no isolated point) that
$$\sup_{G\in\mathcal{B}_0} h(H,G)~=~\sup_{G,T_1,T_2} h(H,\mathcal{G}(G,T_1,T_2))~=~\sup_{X=(X_1,X_2)} h^*(H,X)$$
where in the second expression  $(G,T_1,T_2)$ runs through all triples of finite groups with $T_1,T_2\leq G$ and in the third expression $X=(X_1,X_2)$ runs through all finite joint distributions. 
\end{theorem}

\begin{proof} We have by theorem \ref{thm-inflim1} that the last two quantities coincide. Theorem \ref{main} implies that the first quantity is at least as big as the second one and lemma \ref{lem-inflim1} implies that the second quantity is at least as big as the first one.
\end{proof}

The next corollary establishes Sidorenko's conjecture as a simple entropy inequality involving entropy maximizers. Note that since Sidorenko's conjecture was checked for numerous bipartite graphs corollary \ref{cor-app1} yields a number of new inequalities in information theory. 

\begin{corollary}\label{cor-app1} A bipartite graph $H$ (with no isolated point) satisfies Sidorenko's conjecture if and only if 
$$m(H,X)\geq |E(H)|\mathbb{H}(X)-\sum_{i=1}^2\sum_{v\in V_i(H)}({\rm deg}(v)-1)\mathbb{H}(X_i)$$
holds for every finite joint distribution $X=(X_1,X_2)$.
\end{corollary}

\begin{proof} Using the definition of $h^*$ the inequality is trivially equivalent with $h^*(H,X)\leq |E(H)|$ which is equivalent with Sidorenko's conjecture according to theorem \ref{thm-app1}. 
\end{proof}

The next corollary of theorem \ref{thm-app1} puts Sidorenko's conjecture into a group theoretic context.

\begin{corollary}\label{cor-app2} A bipartite graph $H$ satisfies Sidorenko's conjecture if and only if $$t(H,\mathcal{G}(G,T_1,T_2))\geq t(P_1,\mathcal{G}(G,T_1,T_2))^{|E(H)|}$$ holds for every triple $(G,T_1,T_2)$ where $T_1,T_2$ are subgroups in the finite group $G$. 
\end{corollary}

It is worth mentioning that corollary \ref{cor-app2} implies various known results on Sidorenko's conjecture. For example if $H$ is a tree then trivially $t(H,G)=t(P_1,G)^{|E(H)|}$ holds in any edge-vertex transitive graph and thus corollary \ref{cor-app2} immedieatley implies Sidorenko's conjecture for trees which is not a trivial result. (Note that for paths Sidorenko's conjecture was first proved in a paper by Blackley-Roy in \cite{BR}.)

Another direct implication of corollary \ref{cor-app2} is that if a bipartite graph $H$ is obtained by gluing two graphs $H_1$ and $H_2$ along an edge and $H_1$ and $H_2$ satisfy Sidorenko's conjecture then $H$ also satisfies Sidorenko's conjecture. This was first proved in \cite{LS} but it also follows from the fact that $t(H,G)=t(H_1,G)t(H_2,G)/t(P_1,G)$ holds if $G$ is edge-vertex transitive.

\section{Examples}

\noindent{\bf Convergenet sequences of dense graphs}~Let $\{G_i\}_{i=1}^\infty$ be a convergent graph sequence in  $\mathcal{B}_0$ such that $\lim_{i\to\infty} t(P_1,G_i)>0$. Then it is clear that $\{G_i\}_{i=1}^\infty$ is log-convergent. 

\medskip

\noindent{\bf Hypercubes}~Let us fix $0<\alpha<1$. Let us denote by $G_n$ the bipartite graph on the vertex set $\{0,1\}^n$  in which two vectors are connected if their Hamming distance $d$ is an odd number satisfying $|d/n-\alpha|\leq \epsilon_n$ for some sufficiently slowly decreasing sequence  $\{\epsilon_n\}_{n=1}^\infty$ with $\lim_{n\to\infty}\epsilon_n=0$. We can view $G_n$ as an element in $\mathcal{B}$ by labeling the two color classes with $1$ and $2$.
It can be shown using methods from the present paper that $\lim_{n\to\infty} h(H,G_n)=h^*(H,X)$ holds for every $H\in\mathcal{B}_0$ where in the joint distribution $X=(X_1,X_2)$ both marginals $X_1$ and $X_2$ are uniform on $\{0,1\}$ and  $\mathbb{P}(X_1\neq X_2)=\alpha$.

\medskip

\noindent{\bf Bounded degree graphs}~Let $G_n$ be a growing sequence of graphs in $\mathcal{B}$ with maximum degree $m$ and minimum degree $1$. Assume for simplicity that $|V_1(G_n)|=|V_2(G_n)|=n$. We have that $t(H,G_n)$ is constant times $n^{c(H)-|V(H)|}$ where $c(H)$ denotes the number of connected components in $H$. it follows that the log-limit object is represented by the graph parameter $f(H):=|V(H)|-c(H)$. In other words $f(H)$ is the number of edges in a spanning forest of $H$. Note that $f=R(1/2,1/2)$ and thus $G_n$ is a quasi-random sequence. 

\medskip

\noindent{\bf Projective planes}~Incidence graphs of finite projective planes provide important examples in extremal combinatorics. They are examples for interesting sparse graphs. Let $p$ be a prime number and let $PG(2,p)$ be the projective plane over the prime field $\mathbb{F}_p$. Let $G_p$ denote the incidence graph between points and lines in $PG(2,p)$. We denote by $V_1(G_p)$ the set of points and by $V_2(G_p)$ the set of lines in $PG(2,p)$. We have that $|V_1(G_p)|=|V_2(G_p)|=p^2+p+1$. Furthermore we have that $|E(G_p)|=(p+1)(p^2+p+1)$. This means that $|E(G_p)|$ is roughly of size $|V(G_p)|^{3/2}$. By hand we calculated that $h(H,G_p)$ converges to $R(3/4,1/2,1/2,H)$ for various small graphs $H$.

\begin{question} Is it true that the graphs $G_p$ converge to $R(3/4,1/2)$?
\end{question}

\medskip

\noindent{\bf Heisenberg graphs}~Let $U_p$ denote the Heisenberg group (group of upper uni-triangular matrices in dimension $3$) over the field $\mathbb{F}_p$ with $p$-elements. Let $T_{1,p}$ denote the subgroup of matrices $M\in U_p$ with $M_{1,3}=M_{2,3}=0$ and let $T_{2,p}$ denote the subgroup of matrices $M\in U_p$ with $M_{1,2}=M_{1,3}=0$. Note that $|T_{1,p}|=|T_{2,p}|=p,~|U_p|=p^3$ and $T_{1,p}\cap T_{2,p}=\{1\}$ hold. We call $G_p:=\mathcal{G}(U_p,T_{1,p},T_{2,p})$ the Heisenberg graph over the field $\mathbb{F}_p$. One can calculate that for a connected graph $H\in\mathcal{B}$ the size of the homomorphism set $\Hom(H,G_p)$ is $p$-times the number of maps $f:V(H)\rightarrow\mathbb{F}_p$ with the property that $f(v_1)f(v_2)-f(v_2)f(v_3)+\dots -f(v_n)f(v_1)=0$ holds for every cycle $v_1,v_2,\dots,v_n,v_1$ in $H$. In particular we have that $|\Hom(C_4,G_p)|=p^3(2p-1)$ and thus $t(C_4,G_p)=(2p-1)/p^5$. Using the fact that $t(P_1,G_p)=1/p$ we obtain that $\lim_{p\to\infty}h(C_4,G_p)=4$.

\section{The graph setting and concluding remarks}\label{remarks}

Any graph $G$ can be regarded as a symmetric subset in $V(G)\times V(G)$ and thus it can be represented by a graph in the bipartite setting. More precisely $G$ is represented by the bipartie graph $G'\in\mathcal{B}$ in which the two color classes are identical copies of $V(G)$ and each edge $(v,w)$ of $G$ is represented by two edges $(v,w)$ and $(w,v)$. This representation preserves densities of bipartite graphs. Our results in the bipartite setting can be applied for graphs using this representation. The information theoretic analogue of the graph setting is the study of joint distributions $X=(X_1,X_2)$ where $X_1$ and $X_2$ take values in the same set $F$ and $X$ is symmetric in the sense that $(X_1,X_2)$ has the same distribution as $(X_2,X_1)$. It is important to mention that theorem \ref{main} can be stated for symmetric joint distributions with the stronger conclusion that there is sequence $\{G_i\}_{i=1}^\infty$ of edge-veretex transitive graphs (here edge transitivity means that it is transitive on the directed edges of $G$) such that $\lim_{i\to\infty} h(H,G_i)=h^*(H,G)$ holds for every bipartite graph $H$. (Note that this statement is formulated in the graph setting so $H$ is a normal graph that has no odd cycles.) 

The chapter on quasi-randomness becomes simpler in the graph setting. Recall that in the bipartite limit space quasi-randomness depended on two parameters: $\alpha,\beta$. Since graphs can be represented by bipartite graphs with equal color classes we have that $\alpha=1/2$ always holds and thus we obtain a one parameter family of quasi random objects depending only the sparsity exponent $\beta$. The random graph model corresponding to $\beta$ is a graph $G(n,\beta)$ on $n$ vertices where edges are independently created with probability $n^{2\beta-2}$. It is important that in the graph version of theorem \ref{quasi-thm} the test graphs $H$ are still required to be bipartite since \ref{quasi-lem1} uses this fact heavily.

It is potentially interesting to investigate power relations $d(H_1,G)/d(H_2,G)$ for non bipartite graphs $H_1,H_2$. These quantities are not uniformly bounded and are not necessarily defined since $t(H_i,G)$ can be $0$ even if $G$ is not empty. One can still force compactness by introducing the symbol $\infty$ and regard it as the one point compactification of $\mathbb{R}$. We can also use it if expressions are not defined. In this setting sequences that converge to $\infty$ become formally convergent. Furthermore every sequence $\{G_i\}_{i=1}^\infty$ has a subsequence such that $d(H_1,G)/d(H_2,G)$ is convergent for every pair of graphs $H_1,H_2$.
We are not sure how much from our statements can be saved to this setting.

We finish this chapter with a potential refinement of our convergence notions motivated by information theory.
We mentioned in the introduction that $d(H,G)$ can be interpreted as the relative entropy of the uniform measure on $\Hom(H,G)$ with respect to the uniform measure on all functions $V(G)\rightarrow V(H)$. It is very natural to investigate the relative entropy of a marginal of the uniform measure on $\Hom(H,G)$ on some subset of $V(H)$ in a similar way. This can be formulated as a graph parameter for labeled graphs in which the labels specify the marginal. We can extend the notion of log-convergence with the convergence of all these parameters normalized by $d(P_1,G)$. In a similar fashion we can extend the information theoretic parameters $d^*(H,X)$ to labeled graphs $H$ by regarding mutual information in marginal distributions in the entropy maximizing distributions in $Q(H,X)$.  It is not clear weather these notions are really finer than the original convergence notions however theorem \ref{main} generalizes naturally to these new parameters. 

\bigskip

\noindent{\bf Acknowledgement:}~The research leading to these results has received funding from the European Council under the European Union's Seventh Framework Programme (FP7/2007-2013) / ERC grant agreement n°617747.


\begin{thebibliography}{99}

\bibitem{B}~B. Bollob\'as, {\it Random Graphs}, Cambridge University Press, 2001

\bibitem{BCLSV}~C. Borgs, J.T. Chayes, L. Lov´asz, V.T. S´os, and K. Vesztergombi, {\it Convergent graph sequences
I: Subgraph frequencies, metric properties, and testing}, Adv in Math. 219 (2008), 1801–
1851.

\bibitem{CBCY} C. Borgs, Jennifer T. Chayes, H. Cohn, Y. Zhao, {\it An $L_p$ Theory of sparse graph convergence I: limits, sparse random graph models, and power law distributions}~Arxiv:1401.2906


\bibitem{BR}
G.R. Blakley, P.A. Roy, {A H\"older type inequality for symmetric matrices with nonnegative entries}, {\it Proc. Amer. Math. Soc.} {\bf 16} (1965) 1244-1245 

\bibitem{BS} I. Benjamini, O. Schramm, {\it Recurrence of distributional limits of Finite planar graphs},
Electron. J. Probab. 6 (2001), no. 23, 13


\bibitem{CGW}~F. Chung, R.L. Graham and R.M. Wilson, {\it Quasi-random graphs}, Combinatorica 9 (1989),
345–362.

\bibitem{L} L. Lov\'asz, {\it Large networks and graph limits}, AMS, 2012,  ISBN: 978-0-8218-9085-1

\bibitem{LS} X. Li, B. Szegedy, {\it On the logarithmic calculus and Sidorenko's conjecture}a, to appear

\bibitem{LSz} L. Lov\'asz, B. Szegedy, {\it Limits of dense graph sequences}, J. Combin. Theory Ser. B 96 (2006), no. 6, 933-957. 

\bibitem{LSz2} L. Lov\'asz, B. Szegedy, {\it Szemerédi's regularity Lemma for the analyst}, Geom. Funct. Anal. 17 (2007), no. 1, 252-270.

\bibitem{NP} J. Nesetril, P. Ossona de Mendez, {\it Sparsity}, Springer, 2012

\bibitem{R} A. Razborov, {\it On the Minimal Density of Triangles in Graphs}, Comb., Prob. and Comp. / Volume 17 / Issue 04 / July 2008,

\bibitem{Sid} A.F Sidorenko, {\it A correlation inequality for bipartite graphs}, {\it Graphs Combin.}~{\bf 9} (1993), 201--204

\bibitem{Sz} B. Szegedy, {\it An information theoretic approach to Sidorenko's conjecture}, http://arxiv.org/abs/1406.6738

\end{thebibliography}
\end{document}